\documentclass[12pt, a4paper,oneside]{amsart}
\usepackage{amsmath, amscd, amssymb, amsfonts, stmaryrd, vmargin, amsthm}
\newtheorem{theorem}{Theorem}[section]
\newtheorem{lemma}[theorem]{Lemma}

\newtheorem{Problem}[theorem]{Problem}
\theoremstyle{remark}
\newtheorem*{remark}{Remark}
\numberwithin{equation}{subsection}

\begin{document}
\title[Holomorphic isometric embeddings]{Remarks on holomorphic isometric embeddings between bounded symmetric domains}
\author{Shan Tai Chan}
\curraddr{}
\address{Department of Mathematics, Syracuse University, Syracuse, NY 13244-1150, USA.}
\email{schan08@syr.edu}
\thanks{}

\subjclass[2010]{Primary 53C55, 32H02, 32M15}

\begin{abstract}
In this article, we study holomorphic isometric embeddings between bounded symmetric domains. In particular, we show the total geodesy of any holomorphic isometric embedding between reducible bounded symmetric domains with the same rank.
\end{abstract}
\maketitle
\setcounter{page}{1}
\section{Introduction}
In \cite{Mok12}, Mok established an (algebraic) extension theorem of germs of holomorphic isometric embeddings between bounded symmetric domains with respect to their Bergman metrics up to a normalizing constant.
It is known from the proof of the Hermitian Metric Rigidity Theorem on Hermitian locally symmetric spaces (cf. \cite{Mok89}) that such a holomorphic isometric embedding $f$ is totally geodesic whenever all irreducible factors of the domain of $f$ are irreducible bounded symmetric domains of rank $\ge 2$ (cf. \cite[Theorem 1.3.2]{Mok12}).
Therefore, the focus on the study of holomorphic isometric embeddings between bounded symmetric domains has been restricted to the case where the domain of the embedding is the complex unit ball, i.e., a rank-$1$ bounded symmetric domain (see \cite{CM17, XY16a, XY16b, YZ12}).
In \cite{YZ12}, Yuan and Zhang proved the totally geodesy of any holomorphic isometric embedding from a complex unit $n$-ball, $n\ge 2$, to a product of complex unit balls with respect to the canonical K\"ahler metrics.
Recently, the author, Xiao and Yuan \cite{CXY17} have generalized this result, namely, we have proven the total geodesy of any holomorphic isometric embedding from a product of complex unit $n_j$-balls for $1\le j\le k$, and $n_j\ge 2$ for all $j$, into a product of complex unit balls with respect to the canonical K\"ahler metrics.

In general, there are nonstandard (i.e., not totally geodesic) holomorphic isometric embeddings from a complex unit ball into a bounded symmetric domain of rank $\ge 2$ (see \cite{Ch16, Ch17b, CY17, Mok16, Ng10, UWZ17}).
However, it is expected that under certain assumptions on the domain and target bounded symmetric domains of a holomorphic isometric embedding $f$, we can still obtain some rigidity results for $f$.
In particular, we consider holomorphic isometric embeddings between bounded symmetric domains with the same rank.
In the present article, we denote by $g_D$ the canonical K\"ahler-Einstein metric on an irreducible bounded symmetric domain $D\Subset \mathbb C^n$ normalized so that the minimal disks of $D$ are of constant Gaussian curvature $-2$. We refer the readers to Mok-Tsai \cite{MT92} for the notion of minimal disks of bounded symmetric domains. The main result of the present article is the following.
\begin{theorem}[Main Theorem]\label{HI_TotalGeodesy_EqRk}
Let $D=D_1\times\cdots \times D_k$ and $\Omega=\Omega_1\times\cdots \times \Omega_m$ be bounded symmetric domains in their Harish-Chandra realizations such that $\mathrm{rank}(\Omega)=\mathrm{rank}(D)\ge 2$, where $D_j$, $1\le j\le k$, $\Omega_l$, $1\le l\le m$, are irreducible bounded symmetric domains.
Let $F:D \to \Omega$ be a holomorphic isometric embedding from $(D_1,\lambda_1 g_{D_1})\times\cdots \times (D_k,\lambda_k g_{D_k})$ to $(\Omega_1,\mu_1 g_{\Omega_1})\times\cdots \times (\Omega_m,\mu_m g_{\Omega_m})$ for some positive real constants $\lambda_j$, $1\le j\le k$, and $\mu_l$, $1\le l\le m$.
Then, $F$ is totally geodesic.
\end{theorem}
\begin{remark}
If we further make the assumption that all $D_j$, $1\le j\le k$, are of rank $\ge 2$ in Theorem \ref{HI_TotalGeodesy_EqRk}, then the result follows directly from Mok \cite[Theorem 1.3.2]{Mok12}.
\end{remark}
\section{Preliminary}
In the present article, we will always assume that a given bounded symmetric domain $\Omega\Subset \mathbb C^N$ is in its Harish-Chandra realization.
Let $\Omega=\Omega_1\times \cdots \times \Omega_m\Subset \mathbb C^N$ be a bounded symmetric domain, where $\Omega_j\Subset \mathbb C^{n_j}$, $1\le j\le m$, are irreducible bounded symmetric domains.
Then, it well-known that the Bergman kernel of $\Omega$ is given by
\[ K_\Omega(Z,\xi) = \prod_{j=1}^m K_{\Omega_j}(Z_j,\xi_j), \]
where $K_{\Omega_j}(Z_j,\xi_j)$ is the Bergman kernel of $\Omega_j$ for $1\le j\le m$, $Z=(Z_1,\ldots,Z_m)$ and $\xi=(\xi_1,\ldots,\xi_m)$ for $Z_j,\xi_j\in \mathbb C^{n_j}$, $1\le j\le m$.
In addition, for $1\le j\le m$, we also have
\[ K_{\Omega_j}(Z_j,\xi_j) = C_j h_{\Omega_j}(Z_j,\xi_j)^{-(p(\Omega_j)+2)} \]
for some positive real constant $C_j$ and for some polynomial $h_{\Omega_j}(Z_j,\xi_j)$ in $(Z_j,\overline{\xi_j})$, where $p(\Omega_j)$ is some positive integer depending only on $\Omega_j$ (cf. \cite{CM17}).
Given an irreducible bounded symmetric domain $D\Subset \mathbb C^n$, we let $g_D$ be the canonical K\"ahler-Einstein metric on $D$ normalized so that minimal disks in $D$ are of constant Gaussian curvature $-2$ (cf. \cite{CM17}).
For the notion of the minimal disks, we refer the readers to \cite{MT92}.
Denote by $\omega_{g_D}$ the corresponding K\"ahler form of $(D,g_D)$.
Then, from \cite{CM17} we have 
\[ \omega_{g_D}=-\sqrt{-1}\partial\overline\partial \log h_D(w,w), \]
where $w\in D\subset \mathbb C^n$ are the Harish-Chandra coordinates.
Moreover, we denote by $ds_U^2$ the Bergman metric of any bounded domain $U\Subset \mathbb C^n$.
We have the following lemma, which is one of the consequences of \cite[Theorem 4.25]{CXY17}.
\begin{lemma}(cf. Theorem 4.25 in \cite{CXY17})\label{lem:HI_is_Proper}
Let $F=(F_1,\ldots,F_m):(D_1,\lambda_1g_{D_1})\times\cdots \times (D_k,\lambda_k g_{D_k})\to (\Omega_1,\lambda'_1g_{\Omega_1})\times\cdots \times (\Omega_m,\lambda'_m g_{\Omega_m})$ be a holomorphic isometry, i.e.,
\[ \sum_{j=1}^m \lambda'_j F_j^* g_{\Omega_j}
= \bigoplus_{l=1}^k \lambda_l g_{D_l}, \]
where $D_l$, $1\le l\le k$, $\Omega_j$, $1\le j\le m$, are irreducible bounded symmetric domains in their Harish-Chandra realizations, and $\lambda_l$, $1\le l\le k$, $\lambda'_j$, $1\le j\le m$, are positive real constants.
Then, $F:D_1\times\cdots \times D_k \to \Omega_1\times\cdots \times \Omega_m$ is a proper holomorphic map.
\end{lemma}
\section{Holomorphic isometric embeddings between bounded symmetric domains with the same rank}
Motivated by the Polydisk Theorem (cf. \cite{Mok89, Wo72}), we first study holomorphic isometric embeddings from $\big(\Delta^r,{1\over 2}ds_{\Delta^r}^2\big)$ to $(\Omega,g_\Omega)$ for any irreducible bounded symmetric domain of rank $r\ge 2$.
Then, we obtain the following result, which is a special case of our main result, i.e., Theorem \ref{HI_TotalGeodesy_EqRk}.
\begin{theorem}\label{thm_HI_polydisk_IrrBSD}
Let $F:\big(\Delta^r,{1\over 2}ds_{\Delta^r}^2\big)\to (\Omega,g_\Omega)$ be a holomorphic isometric embedding, where $\Omega\Subset \mathbb C^N$ is an irreducible bounded symmetric domain of rank $r\ge 2$.
Then, $F$ is totally geodesic.
\end{theorem}
\begin{proof}
We prove by induction on the rank $r$ of the target irreducible bounded symmetric domain $\Omega$.

Firstly, we consider the case where $\mathrm{rank}(\Omega)=2$.
By Mok-Tsai \cite[Proposition 2.2]{MT92}, the radial limit $F^*(e^{i\theta},w)$ exists almost everywhere on $\partial\Delta \times \Delta$.
Writing $F_\theta(w):=F^*(e^{i\theta},w)$, Mok-Tsai \cite[pp.\;103-104]{MT92} showed that $F_\theta$ is holomorphic in $w$ by the Cauchy integral formula.
Then, $F_\theta(\Delta)$ lies in a maximal face of $\Omega$, which is biholomorphic to a complex unit ball $\mathbb B^{n(\Omega)}$, where $n(\Omega)$ is the null dimension of $\Omega$ (see \cite[p.\;105]{Mok89}).
As in Mok-Tsai \cite{MT92} and Tsai \cite[Proposition 1.1]{Ts93}, one further deduces from the Fatou's Theorem that
\[ F(z,w) = {1\over 2\pi\sqrt{-1}}\int_{\partial\Delta} {F^*(\xi,w)\over \xi - z} d\xi \]
(cf. Mok \cite{Mok07})
and thus for each $z_0\in \Delta$, $F_{z_0}(\Delta)$ lies in some maximal characteristic symmetric subspace $\Omega'\cong\mathbb B^{n(\Omega)}$ of $\Omega$, where $F_{z_0}(w):=F(z_0,w)$.
Moreover, $(\Omega',g_{\Omega}|_{\Omega'})\cong (\mathbb B^{n(\Omega)},g_{\mathbb B^{n(\Omega)}})$ is of constant holomorphic sectional curvature $-2$.
Then, for each $z_0\in \Delta$, we have a holomorphic isometry
$F_{z_0}:(\Delta,g_\Delta) \to (\Omega',g_{\Omega}|_{\Omega'})\cong (\mathbb B^{n(\Omega)},g_{\mathbb B^{n(\Omega)}})$, which is totally geodesic by the Gauss equation since both the domain and the target are of constant holomorphic sectional curvature $-2$.
Denote by $\sigma$ the $(1,0)$-part of the second fundamental form of $(F(\Delta^2),g_\Omega|_{F(\Delta^2)})$ in $(\Omega,g_\Omega)$.
Write $\Delta^2:=\Delta_1\times \Delta_2$ with $\Delta_j=\Delta$, $j=1,2$.
This shows that for any tangent vector $\alpha\in T_w(\Delta_2) \subset T_{(z,w)}(\Delta^2)$, we have $\sigma(\alpha',\alpha')=0$, where $\alpha':=dF_{(z,w)}(\alpha)$.
Similarly, for any tangent vector $\beta\in T_z(\Delta_1) \subset T_{(z,w)}(\Delta^2)$, we have $\sigma(\beta',\beta')=0$, where $\beta':=dF_{(z,w)}(\beta)$.
Moreover, we have $\sigma(\alpha',\beta')=0$.
These together imply that the second fundamental form $\sigma$ vanishes identically and thus $F$ is totally geodesic. The theorem is proved when the rank of the target irreducible bounded symmetric domain is equal to $r=2$.

Assume that the statement of the theorem holds true when the target irreducible bounded symmetric domain is of rank $r-1\ge 2$ for some $r\ge 3$.
Then, we consider the case where $\Omega$ is of rank $r\ge 3$.
The same arguments imply that for each $z_1\in \Delta$, the map $F_{z_1}:\Delta^{r-1} \to \Omega$ defined by $F_{z_1}(z_2,\ldots,z_r):=F(z_1,z_2,\ldots,z_r)$ has the image lying inside a maximal characteristic symmetric subspace $\Omega'$ of $\Omega$. It is well-known from \cite{Wo72} that $\Omega'$ is an irreducible bounded symmetric domain of rank $(r-1)$.
This induces a holomorphic isometry $F_{z_1}:(\Delta^{r-1},{1\over 2}ds_{\Delta^{r-1}}^2)\to (\Omega',g_{\Omega'})$ for $z_1\in \Delta$.
By the induction hypothesis, $F_{z_1}$ is totally geodesic for each $z_1\in \Delta$.
Applying the similar arguments as before, we see that $F$ is totally geodesic.
\end{proof}

Now, we are ready to prove Theorem \ref{HI_TotalGeodesy_EqRk}.

\begin{proof}[Proof of Theorem \ref{HI_TotalGeodesy_EqRk}]
Write $F:=(F_1,\ldots,F_m)$, where $F_l:D\to \Omega_l$, $1\le l\le m$, are holomorphic maps.
From the assumption, we have
$\sum_{l=1}^m \mu_l F_l^*g_{\Omega_l} = \bigoplus_{j=1}^k \lambda_j g_{D_j}$.
It follows from Lemma \ref{lem:HI_is_Proper} that $F$ is a proper holomorphic map.
For simplicity, we write $(\Omega,g'_\Omega)= (\Omega_1,\mu_1 g_{\Omega_1})\times\cdots \times (\Omega_m,\mu_m g_{\Omega_m})$.
Denote by $\sigma$ the $(1,0)$-part of the second fundamental form of $(F(D),g'_\Omega|_{F(D)})$ in $(\Omega,g'_\Omega)$.

If $\mathrm{rank}(D)=\mathrm{rank}(\Omega)=1$, then $D\cong \mathbb B^n$ and $\Omega\cong \mathbb B^m$ for some positive integers $n$ and $m$.
Note that for any $\lambda>0$ and any integer $N\ge 1$, $(\mathbb B^N,\lambda g_{\mathbb B^N})$ is of constant holomorphic sectional curvature $-{2\over \lambda}$.
Considering the holomorphic isometry $F :(\mathbb B^n,\lambda_1g_{\mathbb B^n})\to (\mathbb B^m,\mu_1g_{\mathbb B^m})$, we have $-{2\over \lambda_1}\le -{2\over \mu_1}$ and thus $0<{\lambda_1\over \mu_1}\le 1$ by the Gauss equation.
On the other hand, it follows from the functional equation for $F$ that ${\lambda_1\over \mu_1}$ is a positive integer (see \cite[Lemma 3]{CM17}).
Therefore, we have ${\lambda_1\over \mu_1}=1$, i.e., $\lambda_1=\mu_1$.
For any unit vector $\alpha\in T_x^{1,0}(F(\mathbb B^n))$ and any $x\in F(\mathbb B^n)$, we have
\[ \lVert \sigma(\alpha,\alpha)\rVert^2
=R_{\alpha\overline\alpha\alpha\overline\alpha}(\mathbb B^m,\mu_1g_{\mathbb B^m})
 - R_{\alpha\overline\alpha\alpha\overline\alpha}(F(\mathbb B^n),\mu_1g_{\mathbb B^m}|_{F(\mathbb B^n)})
= -{2\over \mu_1} - \left(-{2\over \lambda_1}\right) = 0 \]
because $\lambda_1=\mu_1$.
Therefore, we have $\sigma(\alpha,\alpha)=0$.
Then, for any $\alpha',\beta'\in T_x^{1,0}(F(\mathbb B^n))$, $x\in F(\mathbb B^n)$, we have
\[ 0=\sigma(\alpha'+\beta',\alpha'+\beta')
= \sigma(\alpha',\alpha')
+2\sigma(\alpha',\beta')
+\sigma(\beta',\beta') = 2\sigma(\alpha',\beta') \]
so that $\sigma(\alpha',\beta')=0$, i.e., $\sigma\equiv 0$.
Hence, $F$ is totally geodesic. 

On the other hand, if $\mathrm{rank}(D_j)\ge 2$ for $1\le j\le k$, then $F$ is totally geodesic by Mok \cite[Theorem 1.3.2]{Mok12}. From now on, we assume that $D$ is reducible, $\mathrm{rank}(D)=\mathrm{rank}(\Omega)\ge 2$ and some irreducible factor of $D$ is of rank $1$, equivalently biholomorphic to some complex unit ball.
That means we may write $D=\mathbb B^{q_1}\times \cdots \times \mathbb B^{q_{k_1}}\times D_{k_1+1}\times\cdots \times D_k$, where $D_j$, $k_1+1\le j\le k$, are irreducible bounded symmetric domains of rank at least $2$ for some integer $k_1\ge 1$.
For $k_1+1\le j\le k$, there is a totally geodesic complex submanifold $M_j$ of $D_j$ such that $M_j\cong \Delta^{r_j-1}\times \mathbb B^{q_j}$, where $r_j:=\mathrm{rank}(D_j)$ and $q_j$ denotes the complex dimension of the rank-$1$ characteristic symmetric subspace of $D_j$ for $k_1+1\le j\le k$ (cf. \cite{MT92, Ts93}).
Inductively, the arguments in the proof of Theorem \ref{thm_HI_polydisk_IrrBSD} show that the (proper) holomorphic isometry $F$ induces a holomorphic isometry $f:=F|_{\mathbb B^{q_1}\times\cdots \times \mathbb B^{q_{k_1}}\times \mathbb B^{q_{k_1+1}}\times\cdots \times \mathbb B^{q_k}}$ from
$(\mathbb B^{q_1},\lambda_1 g_{\mathbb B^{n_1}})\times\cdots \times (\mathbb B^{q_k},\lambda_k g_{\mathbb B^{q_k}})$ to
$(\Omega^{(k)},g'_{\Omega}|_{\Omega^{(k)}})$ for some rank-$k$ characteristic symmetric subspace $\Omega^{(k)}$ of $\Omega$.

Let $q'_l$ be the complex dimension of the rank-$1$ characteristic symmetric subspace of $\Omega_l$ for $1\le l\le m$.
By the same arguments as in the proof of Theorem \ref{thm_HI_polydisk_IrrBSD}, the map $f$ further induces holomorphic isometries from $(\mathbb B^{q_j},\lambda_j g_{\mathbb B^{q_j}})$ to $(\mathbb B^{q'_{l_j}},\mu_{l_j}g_{\mathbb B^{q'_{l_j}}})$ for some $l_j$ and for any $j$, $1\le j\le k_1$.
Similarly, $f$ induces holomorphic isometries from 
$(\mathbb B^{q_j},\lambda_j g_{\mathbb B^{q_j}})$ to $(\mathbb B^{q'_{l_j}},\mu_{l_j}g_{\mathbb B^{q'_{l_j}}})$ for some $l_j$ and for any $j$, $k_1+1\le j\le k$.
In addition, we have $\lambda_j=\mu_{l_j}$, $1\le j\le k$, by the functional equation and the Gauss equation as before.
Then, these induced holomorphic isometries from 
$(\mathbb B^{q_j},\lambda_j g_{\mathbb B^{q_j}})$ to $(\mathbb B^{q'_{l_j}},\mu_{l_j}g_{\mathbb B^{q'_{l_j}}})$ are totally geodesic by the arguments as before.
More precisely, for $1\le j\le k$, let $\iota_j:\mathbb B^{q_j}\hookrightarrow \mathbb B^{q_1}\times\cdots \times \mathbb B^{q_k}$ be the standard embedding $\iota_j(Z^j)=(Z^1_0,\ldots,Z^{j-1}_0,Z^j,Z^{j+1}_0,\ldots,Z^k_0)$ for fixed $Z^l_0\in \mathbb B^{q_l}$, $1\le l\le k$ and $l\neq j$.
Then, $f\circ\iota_j$ is totally geodesic for $1\le j\le k$ and thus $f$ is totally geodesic.

For each $x=F(Z^1,\ldots,Z^{k_1},W^{k_1+1},\ldots,W^k)\in S:=F(D)$, we have
\[ \sigma(\alpha'_j,\alpha'_j)=0\qquad \forall\;\alpha'_j=dF_{(Z,W)}(\alpha_j),\;\alpha_j\in T_{Z^j}^{1,0}(\mathbb B^{q_j})\subset T_{(Z,W)}^{1,0}(D), \]
for $1\le j\le k_1$, where $(Z,W)$ $=$ $(Z^1,\ldots,Z^{k_1},W^{k_1+1},\ldots,W^k)\in D=\mathbb B^{q_1}\times \cdots \times \mathbb B^{q_{k_1}}\times D_{k_1+1}\times\cdots \times D_k$.
On the other hand, from \cite{Mok12} we have
\[ \sigma(\beta'_j,\beta'_j)=0\qquad \forall\;\beta'_j=dF_{(Z,W)}(\beta_j),\;\beta_j\in T_{W^j}^{1,0}(D_j)\subset T_{(Z,W)}^{1,0}(D), \]
because $\mathrm{rank}(D_j)\ge 2$, $k_1+1\le j\le k$.
Moreover, from the arguments in the proof of \cite[Theorem 1.3.2]{Mok12} we have
\[ \sigma(\eta'_i,\eta'_j)=0 \]
for $i\neq j$, $1\le i,j\le k$, where $\eta'_\mu\in T_{F(Z,W)}^{1,0}(F(D))$ is the image of some tangent vector of the $\mu$-th direct factor $D_\mu$ of $D$ for $1\le \mu\le k$.
Hence, the second fundamental form vanishes identically and thus $F$ is totally geodesic.
\end{proof}

Through the discussion with Professor Wing-Keung To, we actually raised the following problem about the structure of holomorphic isometric embeddings from a reducible bounded symmetric domain to an irreducible bounded symmetric domain of higher rank.

\begin{Problem}\label{Problem_ReBSDtoIRRBSD2}
Let $D=D_1\times\cdots \times D_k$ be a reducible bounded symmetric domain and $\Omega\Subset \mathbb C^N$ be an irreducible bounded symmetric domain of rank $\ge 2$, where $D_j$, $1\le j\le k$, are the irreducible factors of $D$.
Let $F:D \to \Omega$ be a holomorphic isometric embedding from $(D_1, g_{D_1})\times\cdots \times (D_k,g_{D_k})$ to $(\Omega,g_\Omega)$.
Then, does $F(D)$ lie inside some totally geodesic product subdomain $(\Pi,g_\Omega|_{\Pi})\cong (\Omega',g'_{\Omega'}):=(\Omega'_1,g_{\Omega'_1})\times\cdots \times (\Omega'_k,g_{\Omega'_k})$ in $\Omega$ such that $\mathrm{rank}(\Omega')=\mathrm{rank}(\Omega)$?
More precisely, can $F$ be factorized as
\[ F(Z_1,\ldots,Z_k)=\iota(\hat F_1(Z_1),\ldots,\hat F_k(Z_k)) \]
for $(Z_1,\ldots,Z_k)\in D_1\times\cdots \times D_k$, where $\hat F_j:(D_j,g_{D_j})\to (\Omega'_j,g_{\Omega'_j})$, $1\le j\le k$, are holomorphic isometries and $\iota:(\Omega',g'_{\Omega'})\hookrightarrow (\Omega,g_\Omega)$ is the totally geodesic holomorphic isometric embedding?
\end{Problem}
\begin{remark}
Problem \ref{Problem_ReBSDtoIRRBSD2} has been solved by Mok \cite{Mok12} if all irreducible factors $D_j$ of $D$ are of rank $\ge 2$ since the holomorphic isometry $F$ is totally geodesic in this case.
However, some irreducible factor $D_j$ of $D$ could be of rank $1$ in Problem \ref{Problem_ReBSDtoIRRBSD2}, i.e., $D_j\cong \mathbb B^{n_j}$ for some positive integer $n_j$.
In addition, Problem \ref{Problem_ReBSDtoIRRBSD2} is solved when $\Omega$ is of rank $2$ by Theorem \ref{HI_TotalGeodesy_EqRk}.
\end{remark}

\begin{center}
\textsc{Acknowledgements}
\end{center}
The main part of this work was done during the period that the author visited The University of Hong Kong in the summer of 2017.
The author would like to express his gratitude to Professors Ngaiming Mok and Wing-Keung To for helpful discussions on the topic of holomorphic isometric embeddings between (reducible) bounded symmetric domains.


\begin{thebibliography}{XXXXX}
\bibitem[Ca53]{Ca53} Eugenio Calabi:
\emph{Isometric imbedding of complex manifolds},
Ann. of Math. {\bf 58} (1953), 1-23.
\bibitem[Ch16]{Ch16} Shan Tai Chan: \emph{On holomorphic isometric embeddings of complex unit balls into bounded symmetric domains}, Ph.D. thesis at The University of Hong Kong, 2016.
\bibitem[Ch17a]{Ch17a} Shan Tai Chan: \emph{On the structure of holomorphic isometric embeddings of complex unit balls into bounded symmetric domains}, arXiv:1702.01668.
\bibitem[Ch17b]{Ch17b} Shan Tai Chan: \emph{Classification Problem of Holomorphic Isometries of the Unit Disk Into Polydisks}, Michigan Math. J. 66 (2017), 745-767.
\bibitem[CM17]{CM17} Shan Tai Chan and Ngaiming Mok: \emph{Holomorphic isometries of $\mathbb B^m$ into bounded symmetric domains arising from linear sections of minimal embeddings of their compact duals}, Math. Z. 286 (2017), 679-700.
\bibitem[CXY17]{CXY17} Shan Tai Chan, Ming Xiao and Yuan Yuan: \emph{Holomorphic isometries between products of complex unit balls}, Internat. J. Math. 28 (2017), 1740010, 22 pp.
\bibitem[CY17]{CY17} Shan Tai Chan and Yuan Yuan: \emph{Holomorphic isometries from the Poincare disk into bounded symmetric domains of rank at least two}, arXiv:1701.05623.
\bibitem[Lo77]{Lo77} Ottmar Loos:
\emph{Bounded symmetric domains and Jordan pairs}, Math. Lectures, Univ. of California, Irvine, 1977.
\bibitem[Mok89]{Mok89} Ngaiming Mok: \emph{Metric rigidity theorems on Hermitian locally symmetric manifolds},
Series in Pure Mathematics, Vol. 6, World Scientific Publishing Co., Singapore; Teaneck, NJ, 1989.
\bibitem[Mok02]{Mok2002} Ngaiming Mok:
\emph{Local holomorphic isometric embeddings arising from correspondences in the rank-$1$ case},
contemporary trends in algebraic geometry and algebraic topology (Tianjin, 2000), 155-165, Nankai Tracts Math., 5, World Sci. Publ., River Edge, NJ, Singapore, 2002.
\bibitem[Mok07]{Mok07} Ngaiming Mok: \emph{Ergodicity, bounded holomorphic functions and geometric structures in rigidity results on bounded symmetric domains}, in Proceedings of the International Congress of Chinese Mathematicians (Hangzhou 2007), Volume II, Higher Educational Press, Beijing, 2007, 464-505. 
\bibitem[Mok12]{Mok12} Ngaiming Mok: \emph{Extension of germs of holomorphic isometries up to normalizing constants with respect to the Bergman metric}, 
J. Eur. Math. Soc. (JEMS) 14 (2012),
1617-1656.
\bibitem[Mok16]{Mok16} Ngaiming Mok: \emph{Holomorphic isometries of the complex unit ball into irreducible bounded symmetric domains}, Proc. Amer. Math. Soc. 144 (2016),
4515-4525.
\bibitem[MT92]{MT92} Ngaiming Mok and I-Hsun Tsai: \emph{Rigidity of convex realizations of irreducible bounded symmetric domains of rank $\ge 2$}, J. reine angew. Math. 431 (1992), 91-122.
\bibitem[Ng10]{Ng10} Sui-Chung Ng: \emph{On holomorphic isometric embeddings of the unit disk into polydisks}, Proc. Amer. Math. Soc. 138 (2010), 2907-2922.
\bibitem[Ts93]{Ts93} I-Hsun Tsai: \emph{Rigidity of proper holomorphic maps between symmetric domains}, J. Differential Geom. 37 (1993), 123-160.
\bibitem[UWZ17]{UWZ17} Harald Upmeier, Kai Wang and Genkai Zhang: \emph{Holomorphic Isometries from the Unit Ball into Symmetric Domains}, International Mathematics Research Notices, Vol. 2017, 1-35.
\bibitem[Wo72]{Wo72} Joseph A. Wolf: \emph{Fine structures of Hermitian symmetric spaces},
Symmetric spaces (Short Courses, Washington Univ., St. Louis, Mo., 1969-1970), 271-357. Pure and App. Math., Vol. 8, Dekker, New York, 1972.
\bibitem[XY16a]{XY16a} Ming Xiao and Yuan Yuan: \emph{Holomorphic maps from the complex unit ball to Type $\mathrm{IV}$ classical domains}, arXiv:1606.04806.
\bibitem[XY16b]{XY16b} Ming Xiao and Yuan Yuan: \emph{Complexity of holomorphic maps from the complex unit ball to classical domains}, arXiv:1609.07523.
\bibitem[YZ12]{YZ12} Yuan Yuan and Yuan Zhang: \emph{Rigidity for local holomorphic isometric embeddings from $\mathbb B^n$ into $\mathbb B^{N_1}\times\cdots \times \mathbb B^{N_m}$ up to conformal factors}, J. Differential Geometry \textbf{90} (2012) 329-349.
\end{thebibliography}
\end{document}